\documentclass[12pt]{amsart}
\usepackage{amsmath,amssymb,amsthm}
\usepackage{enumitem}
\usepackage{graphicx}
\usepackage{amsfonts}
\usepackage{bigints}
\usepackage{amsfonts}
\usepackage{mathptmx}
\usepackage{hyperref}

\usepackage{epstopdf}
\newcommand{\overbar}[1]{\mkern 1.5mu\overline{\mkern-1.5mu#1\mkern-1.5mu}\mkern 1.5mu}

\theoremstyle{plain}
\newtheorem{theorem}{Theorem}[section]

\newtheorem{corollary}[theorem]{Corollary}
\newtheorem{lemma}[theorem]{Lemma}
\theoremstyle{definition}

\theoremstyle{Conjecture}

\makeatletter
\@addtoreset{equation}{section}
\makeatother

\makeatletter
\newcommand{\Spvek}[2][r]{%
	\gdef\@VORNE{1}
	\left(\hskip-\arraycolsep%
	\begin{array}{#1}\vekSp@lten{#2}\end{array}%
	\hskip-\arraycolsep\right)}

\def\vekSp@lten#1{\xvekSp@lten#1;vekL@stLine;}
\def\vekL@stLine{vekL@stLine}
\def\xvekSp@lten#1;{\def\temp{#1}%
	\ifx\temp\vekL@stLine
	\else
	\ifnum\@VORNE=1\gdef\@VORNE{0}
	\else\@arraycr\fi%
	#1%
	\expandafter\xvekSp@lten
	\fi}
\makeatother
\begin{document}
	
	\title{Lebesgue constants for Cantor sets.}

\author{Alexander Goncharov}
\author{Yaman Paksoy}

	\address{Department of Mathematics, Bilkent University, 06800 Ankara, Turkey}
	\email{goncha@fen.bilkent.edu.tr,\,\,\,yamanpaks@hotmail.com}
	
	\subjclass[2010]{ 41A05 \and 41A44 }
	\keywords{Lebesgue constants, Cantor sets}
\thanks{The research was partially supported by T\"{U}B\.{I}TAK
(Scientific and Technological Research Council of Turkey), Project 119F023.}

\begin{abstract}
 We evaluate the values of the Lebesgue constants in polynomial interpolation for three types of Cantor sets. In all cases, the sequences of  Lebesgue constants are not bounded. This disproves the statement by Mergelyan.
 \end{abstract}
		\maketitle

	\section{\bf Introduction}

An analysis of the Lebesgue constants $\Lambda_N$ is of fundamental importance in Approximation Theory.
By Lebesgue's Lemma (see, e.g., \cite{CA}, Chapter 2, Prop. 4.1), the accuracy of approximating functions by interpolation is closely related to the size of
$\Lambda_N $.

Given compact set $K \subset \Bbb R$, we consider an array $X=(X_N)_{N=1}^{\infty}=(x_{k,N})_{k=1, N=1}^{N, \infty}$ of interpolating nodes from $K$
and the corresponding Lebesgue constants $\Lambda_N(X_N),\  N=1,2,\ldots,$ (see Section 2 for the definitions).
We say that $K$  belongs to a family ${\mathcal{BLC}}$ (Bounded Lebesgue constants) if there
is an array $X$ of points from $K$ such that the sequence $(\Lambda_N(X_N))_{N=1}^{\infty}$ is bounded. By the classical Faber paper \cite{F},
an interval does not belong to ${\mathcal{BLC}}$. This result  was extended by Erd\"os and P. V\'ertesi  \cite{EV} (see also Cor. 7 in \cite{BDP}): If $K$ belongs to ${\mathcal{BLC}}$, then $K$ is nowhere dense and its one-dimensional Lebesgue measure is zero.

For countable sets, Obermaier proved in  \cite{OB} that geometric progressions belong to ${\mathcal{BLC}}$, whereas Privalov presented in  \cite{PR} a countable set outside ${\mathcal{BLC}}$.

We are interested in the problem: are there perfect sets in the class ${\mathcal{BLC}}$? By \cite{EV}, if such a set exists, it must be of Cantor type, with zero Lebesgue measure. Usual references here are Korovkin  \cite{KOR} and Mergelyan  \cite{M}.
Korovkin constructed a perfect set $K$ and an array $X$ of points from $K$ such that the subsequence $(\Lambda_{N^2}(X_{N^2}))_{N=1}^{\infty}$ is bounded.
In Sections 4 and 6 we present Cantor sets with bounded subsequences of the Lebesgue constants. This allows us to approximate continuous functions
on such sets with the help of Lagrange interpolation polynomials (see e.g. \cite{SV}, p. 161), but does not guarantee the convergence of the Newton interpolation
process, which includes polynomials of all degrees and, accordingly, the existence of Faber interpolating bases (for details, see e.g. \cite{BDP}).

In 1951, Mergelyan published the book \cite{M} containing his classical results in the theory of complex approximation. In one of his supplementary theorems,
the author asserted that some Cantor sets belong to ${\mathcal{BLC}}$. We consider his proof in Section 3 and show that it is not correct. Moreover, we put forward the following\\

{\bf Conjecture}. There are no perfect sets in the class ${\mathcal{BLC}}$.\\

In support of this conjecture,  we present in Section 5 Cantor sets outside ${\mathcal{BLC}}$. Our results are consistent with numerical experiments from \cite{GGP}.

\section{\bf Notations}

Let $K$ be a perfect (i.e. closed with no isolated points), bounded subset of $\Bbb R$ and $X=(x_{k,N})_{k=1, N=1}^{N, \infty}$ be an infinite triangular matrix of points from $K$ such that
each row $X_N$ consists of distinct elements. For a fixed $N$, the points of $X_N$  determine the polynomial $\omega_{N}(x) = \prod_{k=1}^{N}(x-x_{k,N})$,
the fundamental Lagrange polynomial $l_{k,N}(x)= \frac{\omega_{N}(x)}{(x-x_{k,N})\omega'_{N}(x_{k,N})}$, the Lebesgue function
$\lambda_N(x)=\sum _{k=1}^{N}|l_{k,N}(x)|$ and the Lebesgue constant $\Lambda_N(X_N, K)=\sup_{x\in K} \lambda_N(x).$ Given function $f$
defined on $K$, by $L_N(f,x ;X_N) = \sum _{k=1}^{N}f(x_{k,N}) l_{k,N}(x)$
we denote the corresponding Lagrange interpolating polynomial, by $L_N(\cdot, X_N)$ the interpolating projection from $C(K)$ (the space of continuous functions on $K$) to its finite dimensional subspace of polynomials of degree at most $N-1$.
It is easy to show that $\Lambda_N(X_N, K)$ is the operator norm of the projection $L_N(\cdot, X_N)$.\\

{\bf Remarks. 1.} We follow notations from \cite{SV}, where the index $N$ of the Lebesgue constant corresponds to the number of interpolating points,
but not \cite{RIV}, where $N$ is the degree of the interpolating polynomial. Thus, here, the degree of $l_{k,N}$ and $L_N(f)$ is $N-1.$\\
{\bf 2.} Our definition of the class ${\mathcal{BLC}}$ is different from the analogous definition in \cite{BDP}, where any arrays of points from $[a,b]$
are allowed, where $[a,b]$ is the minimal interval containing $K$. We consider a more restricted class, since our main interest lies in the existence of Faber
interpolation bases in various spaces of functions defined on $ K $.\\

In all sections, except Section 6, we consider geometrically symmetric Cantor sets, where, during the Cantor procedure, all intervals of the same level have
the same length. Let $(\ell_{s})_{s=0}^{\infty}$ be a sequence such that $ \ell_{0}=1$ and
\begin{equation}\label{l}
3 \ell_{s+1} \,\leq\,\ell_{s}\,\,\,\,\mbox{for} \,\,\,\, s \in \Bbb N_0.
 \end{equation}
Let $ K$ be the Cantor set associated with this sequence, that is $ K=\bigcap_{s=0}^{\infty} E_{s},$ where $E_{0}=I_{1,0}= [0,1],\,E_{s}$ is a union of
$2^{s}$ closed {\it basic} intervals $I_{j,s}$ of length $\ell_{s}$ and $E_{s+1}$ is obtained by replacing each $I_{j,s}\,\,,j=1,2,...2^{s},$ by two
{\it adjacent} subintervals $I_{2j-1,s+1}$ and $I_{2j,s+1}$. Let $h_{s}=\ell_{s}-2\,\ell_{s+1}$ be the distance between them.
By \eqref{l}, $h_s \geq \frac{1}{3} \ell_s$ for each $s \in \Bbb N_0$.\\

Let $Y_0:=\{0, 1\}$ and, for $k\in \Bbb N,$ let $Y_k$ be the set of endpoints of intervals from $E_k$.
For example, $Y_1:=\{0,\ell_1, 1-\ell_1 , 1\}, \ Y_2:= Y_1 \cup \{\ell_2, \ell_1-\ell_2, 1-\ell_1+\ell_2, 1-\ell_2\}.$   Thus
$\#(Y_s)=2^{s+1}.$ Here and below, $\#(Z)$ denotes the cardinality of a finite set $Z$.

Suppose we are given a set $Z=(z_k)_{k=1}^{N}\subset K$. Let $m_{j,s}(Z):=\#(Z\cap I_{j,s}).$ We say that an interval
$I_{j,s}$ is {\it empty} if  $m_{j,s}(Z)=0.$ We also say that points are {\it uniformly distributed} on $K$  and denote this by $Z\in \mathcal{U}$ if
for each $k\in \Bbb N$ and $i,j \in \{1,2, \ldots, 2^k\}$ we have
$$|m_{i,k}(Z)- m_{j,k}(Z)|\leq 1.$$
 Thus, if $2^{s-1}\leq N< 2^s$ and points of $Z$ are uniformly distributed, then $m_{i,s}(Z)\in \{0, 1\}$ for each $i$. 

Also, for a fixed $x\in \Bbb R$, by $d_k(x,Z)$ we denote the distances $|x-z_{j_k}|$
from $x$ to points of $Z$, where these distances are arranged in the nondecreasing order, so that $d_i(x,Z) \leq d_{i+1}(x,Z)$ for $i = 1,2,...,N-1$. In what follows, we omit the argument $ Z $ in $m_{j,s}(Z), d_k(x,Z)$, as well as $N$ in
$l_{k,N}$ if this does not cause misunderstanding.\\

For a fixed set $Z$ and a basic interval $I_{j,q}$ with $m_{j,q}\geq 2,$ let
\begin{equation}\label{R}
R_{j,q} :=\max\{R:\,\,\,\,\mbox{there exists} \,\,\,\,I_{i,R}\subset I_{j,q} \,\,\,\,\mbox{with} \,\,\,\, m_{i,R}=2\}.
 \end{equation}

Thus, $R_{j,q}\geq q$ and $m_{j,q}\leq 2^{R_{j,q}-q+1}$ since $I_{j,q}$ contains $ 2^{R_{j,q}}$ subintervals of the $R_{j,q}-$th level.
Similarly, if $2^k+1\leq m_{j,q}\leq 2^{k+1}$ then $R_{j,q}\geq q+k.$

\begin{lemma}\label{RR}
Suppose $Z=(z_k)_{k=1}^{2^s-1}\in \mathcal{U}.$ Then $R_{1,0}=s-1.$ If $Z\notin \mathcal{U}$ then $R_{1,0}\geq s.$
\end{lemma}
\begin{proof}
The interval $I_{1,0}$ contains $2^s$ subintervals of the $s-$th level. Hence at least one of them is empty. If $Z\in \mathcal{U}$ then
$m_{i,s}\leq 1$ for all $i$. Hence, $R_{1,0}\leq s-1.$ There are $2^{s-1}$ subintervals of the $s-1-$st level with $m_{j,s-1}\leq 2$ for all $j.$
Clearly, at least one $I_{j,s-1}$ contains two points of $Z$. Thus the value $R$ corresponding to the whole set $K$ must be $s-1.$

On the other hand, if $Z\notin \mathcal{U}$ then $R_{1,0}\geq s.$ In the contrary case, let $R_{1,0}\leq s-1.$ The values $R_{1,0}\leq s-2$
are impossible since $2^R$ intervals of $R-$th level can contain at most $2^{R+1}$ points of $Z$. Hence $R_{1,0}=s-1.$ Let $k_j$ with
$0\leq j \leq 2$ be the  number of intervals of the $s-1-$st level containing $j$ points of $Z$. Then $k_0+k_1+k_2=2^{s-1}$ and $k_1+2 k_2=2^s-1.$
It follows that $k_2=2^{s-1}-k_0-k_1$ and $2^s-1=2^s-k_1-2k_0.$ Hence $k_1+2 k_0=1$ with $k_0=0, k_1=1,$ which means $Z\in \mathcal{U},$ a contradiction.
\end{proof}

Given $x\in K\cap I_{j,R}$ we will use the chain of basic intervals containing it:
\begin{equation}\label{x}
x\in I_{j,R} \subset  I_{j_1,R-1} \subset I_{j_2,R-2} \subset \cdots \subset I_{j_R,0}=[0,1].
\end{equation}
Let $J_{R}$ and $I_{j,R}$ be the  adjacent subintervals of $I_{j_1,R-1}$ and, more generally, $J_n:=(I_{j_{R-n+1},n-1}\setminus I_{j_{R-n},n})\cap E_n$
for $1\leq n \leq R-1.$ If, in addition, the set $Z$ be given, then $\nu_n:=\#(J_n \cap Z).$\\

We consider three types of sets. The first is $K_{\beta}$ with $0<\beta \leq 1/3$, denoting the set given by the sequence $\ell_s=\beta \ell_{s-1}$ for
$s\in \Bbb N.$ Thus,  $\ell_s=\beta^s$ and $K_{1/3}$ stands for the classical Cantor ternary set.

The second type of sets is $K^{\alpha}$ with $1<\alpha$, $\ell_1\leq 1/3$ satisfying $\alpha \geq (\log \ell_1 - \log3) / \log \ell_1$ and $\ell_{s+1}= \ell_s^\alpha = \ell_1^{\alpha^s}$ for  $s \in \Bbb N.$ 

The third family of sets will be given in Section 6.

We follow \cite{GS} to define $\ell-regular$ Cantor sets as sets for which
$$ \ell^2_{s+1}\geq \ell_{s}\ell_{s+2}\,\,\,\,\,\,\mbox{for \,\,\,} s \in \Bbb N. $$
For such sets the value $\frac{h_s}{\ell_s}$ increases. In particular, the sets  $K^{\alpha}$ are $\ell$-regular.\\

Let $[a]$ denote the greatest integer in $a, \log$ stand for natural logarithm.

\section{\bf Sets $K_{\beta}$ and Mergelyan's result}

The first choice of interpolation nodes that comes to mind is the set $Y_{s-1}$ consisting of all endpoints of intervals in Cantor procedure
up to the level $s-1$. However, in the cases when $K=K_{\beta}$ or $K=K^{\alpha}$ with $\alpha<2$, the sequence $(\Lambda_{2^s}(Y_{s-1}, K))_{s=1}^{\infty}$
is not bounded. What is more, only one fundamental polynomial at a certain point takes as large values as we wish for large $s$.

The next lemma is valid for any geometrically symmetric Cantor set $K$. We arrange points from $ Y_{s-1}$ in ascending order.
Then $x_1=0, x_2=\ell_{s-1},\ldots, x_{2^s}=1.$ We select $k=2^{s-1}-1$ with $x_k=\ell_{1}-\ell_{s-1}.$
Also, let ${\tilde x}:=\ell_s.$

\begin{lemma}\label{Y}
Given $s\geq 3,$ let $Y_{s-1}$ be the interpolating set for $K$ and $k, {\tilde x}$ be as above. Then
$$|l_k({\tilde x})|\geq \ell_s\,\left( \frac{1-\ell_s}{1-\ell_1+ \ell_{s-1}}\right)^{2^{s-1}}.$$
\end{lemma}
\begin{proof}

We have $$ | l_k ({\tilde x})| = \prod \limits_{\substack{j=1 \\  j \neq k}}^{2^{s-1}} \left| \dfrac{{\tilde x} - x_j}{x_k - x_j}\right|\cdot
 \prod \limits_{j=2^{s-1}+1}^{2^{s}} \left|  \dfrac{{\tilde x} - x_j}{x_k - x_j} \right|=: \pi_1\cdot \pi_2.$$
Let us obtain the lower bounds of these two separately.
Notice that $\pi_1$ corresponds to the product of ratios of distances of $\tilde x$ and $x_k$ to the nodes from the set $Y_{s-1}\cap I_{1,1}\setminus\{x_k\}.$
This set contains $2^{s-1}-1$ points and for the corresponding $d_j(\tilde x)$ we have: $d_1(\tilde x)=\ell_{s}, \ d_2(\tilde x)=\ell_{s-1}-\ell_{s}$.

For $2\leq j\leq 2^{s-1}-3$ we have $d_{j+1}(\tilde x)=d_j(x_k)+\varepsilon$ with $\varepsilon=\ell_{s-1}-\ell_{s}.$ Indeed, for such $j$ we have $d_j(x_k)=d_j(x_2)$ and $x_2-\tilde x=\varepsilon.$
Also, $d_1(x_k)=\ell_{s-1}, \ d_{2^{s-1}-2}(x_k)=x_k-x_2=\ell_1-2\ell_{s-1}, \  d_{2^{s-1}-1}(x_k)=\ell_1-\ell_{s-1}$ and $d_{2^{s-1}-1}(\tilde x)=\ell_1-\ell_{s}.$
Therefore,
\begin{equation}\label{pi1}
\pi_1= \frac{d_1(\tilde x)\,d_2(\tilde x)}{d_1(x_k)} \cdot \left(\prod_{j=2}^{2^{s-1}-3}\frac{d_{j+1}(\tilde x)}{d_j(x_k)} \right) \cdot  \frac{d_{2^{s-1}-1}(\tilde x)} {d_{2^{s-1}-2}(x_k) d_{2^{s-1}-1}(x_k)}.
\end{equation}

We neglect the product in the middle, as all its terms are greater than one. Hence,
$$ \pi_1>  \frac{\ell_{s}(\ell_{s-1}-\ell_{s})}{\ell_{s-1}} \cdot\frac{\ell_1-\ell_{s}}{(\ell_1-2\ell_{s-1})(\ell_1-\ell_{s-1})}>
\frac{\ell_{s}(\ell_{s-1}-\ell_{s})}{ \ell_{s-1}\,\ell_1}.$$
The last fraction exceeds $\ell_{s},$  as is easy to check.

As for $\pi_2,$ that is the product of ratios of distances to the nodes in $I_{2,1}$, we have $d_j(\tilde x) = d_j(x_k) + x_k-\tilde x$ for
$j = 2^{s-1}+1,...,2^{s}$.
Here, $x_k-\tilde x=\ell_1-\ell_{s-1}-\ell_{s}$ and $d_j(x_k)\leq 1-x_k=1-\ell_1+\ell_{s-1}.$ 
Hence,
\begin{equation}\label{pi2}
 \pi_2=\prod_{j=2^{s-1}+1}^{2^{s}} \left(1+\frac{x_k-\tilde x}{d_j(x_k)}\right) \geq \prod_{j=2^{s-1}+1}^{2^{s}}
\left(1+\frac{\ell_1-\ell_{s-1}-\ell_{s}}{1-\ell_1+\ell_{s-1}}\right)= \left( \frac{1-\ell_{s}}{1-\ell_1+ \ell_{s-1}}\right)^{2^{s-1}}.
\end{equation}
Combining these two together, we get the result.
\end{proof}

\begin{theorem}\label{beta}
For any set  $K_{\beta}$ with $0<\beta \leq 1/3$ we have $\Lambda_{2^{s}}(Y_{s-1}, K_{\beta})\to \infty$ as $s \to \infty.$
\end{theorem}
\begin{proof}
It is easy to check that $(1-\beta^{s})(1-\beta^2)> 1-\beta+\beta^{s-1}$ for $s\geq 3$. Applying Lemma \ref{Y} for $\ell_s=\beta^s$ yields
$|l_k({\tilde x})|\geq \frac{\beta^{s}}{(1-\beta^2)^{2^{s-1}}}\to \infty$ as $s\to \infty.$
 Of course, $\Lambda_{2^s}(Y_{s-1}, K_{\beta})> |l_k({\tilde x})|.$
\end{proof}

Theorem \ref{beta} contradicts Theorem 6.2 from \cite{M}, where the author considered interpolation of functions at the endpoints of intervals of the Cantor procedure. For the convenience of the reader, we give the last result in our terms. Let $w_f$ be the modulus of continuity of a continuous function $f$.\\

 {\bf Theorem 6.2} (\cite{M}) Let $K$ be a geometrically symmetric Cantor set with $h_s>l_{s+1}$ for each $s$. There exists a positive function $\varphi_K(n)$ of integer argument such that, for every function $ f $ continuous on $ K $, the inequality
\begin{equation}\label{merg}
 \max_{x\in K}|f(x)-L_{2^{s+2}}(f,x;Y_{s+1})|< C\,w_f(\varphi_K(2^{s+2}))
 \end{equation}
holds, where $C$ does not depend on $s$.\\

The set $K_{\beta}$ with $\beta<1/3$ satisfies the condition of the theorem. Since the right side of \eqref{merg} is bounded by $2C$ for all continuous
functions $f$ with $||f||\leq 1$, it follows that the Lebesgue constants $\Lambda_{2^{s+2}}(Y_{s+1}, K_{\beta})$ are uniformly bounded. Indeed, for a fixed $s$,
we take a point
${\overbar x}\in K_{\beta}$ for which $\Lambda_{2^{s+2}}(Y_{s+1}, K_{\beta})=\lambda_{2^{s+2}}({\overbar x})$ and a function $f$ with $||f||=1$ such that
$f(x_k)=sign \,l_k({\overbar x})$, so that we have $\Lambda_{2^{s+2}}(Y_{s+1}, K_{\beta}) = L_{2^{s+2}}(f,\overbar x;Y_{s+1})$.

 From $||f||=1$, it follows that $w_f(\varphi_K(2^{s+2})) \leq 2$. Hence, we get $$ |f(\overbar x)-\Lambda_{2^{s+2}}(Y_{s+1}, K_{\beta})|  =|f(\overbar x)-L_{2^{s+2}}(f,\overbar x;Y_{s+1})|< C\,w_f(\varphi_K(2^{s+2})) \leq 2C.$$

This implies $$ \Lambda_{2^{s+2}}(Y_{s+1}, K_{\beta}) \leq 2C +1.$$
Since $C$ does not depend on $s$, this is contradictory to Theorem \ref{beta} that he have just proven.

 The wrong estimation in \cite{M} is on the page 65, 4-th line from below, where the author estimated above the large number
 $M_s:=\frac{2^{2s+3} \ell_{s+2}}{(2^{s+1}!)\ell_{s+1}^{2^{s+1}}}$ by a bounded value $C\,w_f(M_s).$ For $K_{\beta}$, we have $\ell_{s}=\beta^s,$ which is
 $\Delta_{s-1}$ in the author notation. By Stirling’s formula, the leading term of $\log M_s$ is $2^{s+1}\, (s+1)\,(\log1/{\beta}-\log 2),$ which tends
 to $\infty$ as $s\to \infty$ since $\beta<1/3.$\\

Theorem \ref{beta} does not imply $K_{\beta}\notin {\mathcal{BLC}}$. We believe that the choice of interpolating nodes $Y_s$ here is far from optimal.
In the case of small Cantor sets, this choice can be applied.

\section{\bf Sets $K^{\alpha}$. Uniform distribution.}

Suppose $\alpha >1$ is fixed. The set $K^{\alpha}$ is associated with the sequence $\ell_s=\ell_1^{\alpha^{s-1}}$ for $s\geq 1.$
The conditions \eqref{l} with $s=0$ and $s=1$ imply $\ell_1\leq \min\left \{\frac{1}{3}, \left(\frac{1}{3}\right)^{\frac{1}{\alpha-1}}\right\}.$

\begin{lemma}\label{sum}
For each $\alpha > 1$ there are $C_{\alpha}, n_{\alpha}\in \Bbb N$ such that $\sum_{k=0}^n \frac{1}{2^k\,h_k}\leq \frac{C_{\alpha}}{2^n\,h_n}$
for $n\geq 0$ and $\sum_{k=0}^n \frac{1}{2^k\,h_k}\leq \frac{2}{2^n\,h_n}$ for $n\geq n_{\alpha}$.

If  $\alpha \geq 2$ then $C_{\alpha}=7, n_{\alpha}=4.$ For $1<\alpha <2$ we can take $C_{\alpha}=5$ and
$n_{\alpha}=3+\left[\frac{\log(\log 12/\log 3)}{\log \alpha}\right].$
\end{lemma}
\begin{proof}
Assume the bound $\sum_{k=0}^n \frac{1}{2^k\,h_k}\leq \frac{A_n}{2^n\,h_n}$  is valid for some $n\geq 0.$
We aim to express $A_{n+1}$ in terms of $A_n.$ Since $\frac{A_n}{2^n\,h_n}+\frac{1}{2^{n+1}\,h_{n+1}}\leq \frac{A_{n+1}}{2^{n+1}\,h_{n+1}},$ we have
\begin{equation}\label{A}
A_{n+1}=1+2 \, A_n \cdot \frac{h_{n+1}}{h_n}.
\end{equation}

We first examine the case $\alpha \geq 2$. For $n=0$ and $n=1$ we use the estimate $h_{n+1}<h_n.$ It is not rough, since with $\ell_1=1/3$ we have
$h_{n+1}\sim h_n$ for large  $\alpha$, whereas $\frac{h_2}{h_1}= \frac{7}{9}$ for $\alpha=2$. The obvious value $A_0=1$ gives $A_1=3$ and $A_2=7.$

If $n\geq 2$ then $\frac{h_{n+1}}{h_n}<\frac{\ell_{n+1}}{\ell_n-2\ell_{n+1}}=\frac{1}{\ell_n^{1-\alpha}-2}\leq \frac{1}{\ell_n^{-1}-2}\leq \frac{1}{3^{\alpha^{n-1}}-2},$ since $\alpha \geq 2$ and $\ell_1\leq 1/3.$ Hence,
$A_{n+1}=1+2 \, A_n \cdot (3^{2^{n-1}}-2)^{-1}$. From this $A_3=3$ and $A_n<2$ for $n\geq 4.$\\

We now turn to the case $1<\alpha <2$. Here we take $\ell_1 = \left(\frac{1}{3}\right)^{\frac{1}{\alpha-1}}$ so
$h_n=\ell_n(1-2\ell_n^{\alpha-1})= \ell_n(1-2\cdot 3^{-\alpha^{n-1}})$ for $n\geq 1.$ Hence,
$$ \frac{h_{n+1}}{h_n}= \frac{\ell_{n+1}(1-2\cdot 3^{-\alpha^{n}})}{\ell_n(1-2\cdot 3^{-\alpha^{n-1}})}=
\frac{1-2\cdot 3^{-\alpha^{n}}}{3^{\alpha^{n-1}}-2}.$$
For a fixed $n,$ let $f(\alpha)=f_n(\alpha)=\frac{1-2\cdot 3^{-\alpha^{n}}}{3^{\alpha^{n-1}}-2}.$ A direct computation shows that
$f'(\alpha)<0$ provided $n\geq 3.$
Indeed, $f'(\alpha)<0 \Leftrightarrow 2 n \alpha (3^{\alpha^{n-1}}-2)< 3^{\alpha^{n-1}}\, (3^{\alpha^{n}}-2)(n-1),$ which is valid for $n\geq 3.$ Hence,
\begin{equation}\label{hh}
\frac{h_{n+1}}{h_n}<f(1)=\frac{1}{3}\,\,\,\,\,\,\mbox {if}\,\,\,\,\,\, n\geq 3.
\end{equation}

Let us calculate the first values of $A_n: A_0=1, A_1=1+2\frac{h_1}{h_0}$ with
$\frac{h_1}{h_0}=\frac{\ell_1(1-2\cdot \ell_1^{\alpha-1})}{1-2 \ell_1}=\frac{1}{3} \frac{\ell_1}{1-2 \ell_1}< \frac{1}{3},$ so we can take
$A_1=\frac{5}{3}.$ Similarly, $\frac{h_2}{h_1}=\frac{\ell_2(1-2\cdot \ell_2^{\alpha-1})}{\ell_1(1-2\cdot \ell_1^{\alpha-1})}=1-2\cdot 3^{-\alpha},$
which is smaller than $\frac{7}{9}$ for $1<\alpha <2$. Hence, $A_2=\frac{97}{27}.$ Let us show that $\frac{h_3}{h_2}<\frac{1}{2}.$
This will give $A_3=5.$ Here, $\frac{h_3}{h_2}=f_2(\alpha)=3^{-\alpha^2}\, \frac{3^{\alpha^2}-2}{3^{\alpha}-2}<\frac{1}{2}$ if
$4\cdot 3^{\alpha}< 3^{\alpha^2+\alpha}+4.$ It is a simple matter to check that it is valid for given $\alpha.$

In its turn, $\frac{h_4}{h_3}<\frac{1}{3},$ since $5\cdot 3^{\alpha^3}< 3^{\alpha^3+\alpha^2}+6.$ The last inequality is valid, because the function
$F(\alpha)=3^{\alpha^3+\alpha^2}-5\cdot 3^{\alpha^3}+6$ is increasing and positive on the interval $(1,2).$

By \eqref{hh}, for $n\geq 3$ we can take $A_{n+1}=1+\frac{2}{3} \, A_n$. This yields $A_4=\frac{13}{3},\, A_5=\frac{35}{9},$ etc. We see that
$A_n<4$ for $n\geq 5,$ so $C_{\alpha}$ can be taken as 5.

We now turn to \eqref{A}, which implies $A_n<1+2\cdot 5\cdot \frac{\ell_n}{\ell_{n-1}-2\ell_n}=1+10(3^{\alpha^{n-2}}-2)^{-1}.$ It
follows that $A_n<2$ if $n\geq 3+\left[\frac{\log(\log 12/\log 3)}{\log \alpha}\right].$
\end{proof}

\begin{lemma}\label{llh}
Let $\alpha > 1$ and $k\geq 2.$ Then $\frac{\ell_k}{h_{k-1}}\leq (3^{\alpha^{k-2}}-2)^{-1}, \frac{\ell_k}{h_k}\leq 1+2\cdot (3^{\alpha^{k-1}}-2)^{-1}.$
\end{lemma}
\begin{proof} Indeed, $\frac{\ell_{k-1}}{\ell_k}={\ell_1}^{-(\alpha-1) \alpha^{k-2}}\geq 3^{\alpha^{k-2}}$ and $\frac{\ell_k}{h_{k-1}}=\frac{1}{\frac{\ell_{k-1}}{\ell_k}-2},\,\, \frac{\ell_k}{h_k}=1+ \frac{2\ell_{k+1}}{h_k}.$
\end{proof}

\begin{lemma}\label{llq}
Suppose that the points $(x_k)_{k=1}^{2^s}$ with $x_1<x_2<\cdots <x_{2^s}$ are uniformly distributed on $K^{\alpha}$.
Then, for each $q$ with $0\leq q \leq s,$ we have $\ell_q\leq x_j+x_{2^{s-q}+1-j}+\ell_s$ for $1\leq j \leq 2^{s-q}.$
\end{lemma}
\begin{proof} The interval $I_{1,q}$ contains $ 2^{s-q}$ subintervals of the $s-$th level with $x_j\in I_{j,s}$.
The intervals $I_{j,s}$ and $I_{2^{s-q}+1-j,s}$ are symmetric with respect to
$\ell_q/2.$ Therefore, the distance from $x_{2^{s-q}+1-j}$ to $\ell_q$ does not exceed $x_j+\ell_s.$
\end{proof}

We proceed to show that, for small sets, the choice of interpolating nodes $Y_{s-1}$ can ensure boundedness of a subsequence of Lebesgue constants,
compare to \cite{KOR}.

\begin{theorem}\label{bdd1}
Suppose that for each $s$ the set $Z_s=(x_{k,s})_{k=1}^{2^s}$ is uniformly distributed on $K^{\alpha}$. If $\alpha \geq 2$ then
$\Lambda_{2^s}(Z_s, K^{\alpha})\to 1$ as $s\to \infty$.
If $\alpha < 2$ then $\Lambda_{2^s}(Z_s, K^{\alpha})\to \infty$ as $s\to \infty$.
\end{theorem}

\begin{corollary}
The sequence $(\Lambda_{2^s}(Y_{s-1}, K^{\alpha}))_{s=1}^{\infty}$ is bounded if and only if $\alpha \geq 2$.
\end{corollary}
\begin{proof}

Consider first $\alpha < 2.$ We begin by extending Lemma \ref{Y} to any uniform distribution. Fix $Z_s\in \mathcal{U}.$
For brevity, we drop the subscript $s$ in $x_{j,s}$, which are numbered in order of increase, so $x_j\in I_{j,s}.$
Fix ${\tilde x}\in I_{1,s}$ in a such way that ${\tilde x}$ and $x_1$ lie on different subintervals of the $s+1$-st level.
As in Lemma \ref{Y}, we estimate $ | l_k ({\tilde x})|= \pi_1\cdot \pi_2$ for $k=2^{s-1}-1.$

Here,  $d_1(\tilde x)\geq h_{s}$ and $d_j(\tilde x)=x_j-\tilde x$ for $j\geq 2.$ In its turn, $d_1(x_k)=x_{k+1}-x_k\leq \ell_{s-1},
d_j(x_k)=x_k-x_{k+1-j}$ for $j\geq 2$ with $d_{k-1}(x_k)\leq \ell_1-2\ell_{s-1}+2\ell_s, d_k(x_k)<\ell_1.$

Let us show that, as in Lemma \ref{Y}, $d_j(x_k)\leq d_{j+1}(\tilde x)$ for $2\leq j\leq k-2.$
Indeed, Lemma \ref{llq} with $q=1$ and $j+1$ instead of $j$ yields $\ell_1\leq x_{j+1}+x_{k+1-j}+\ell_s$. Then
$d_j(x_k)=\ell_1-x_{k+1-j}-(\ell_1- x_k)\leq x_{j+1}+\ell_s-(\ell_1- x_k)=d_{j+1}(\tilde x)+ \ell_s+\tilde x-(\ell_1- x_k)\leq d_{j+1}(\tilde x),$
by \eqref{l}, since  $\ell_s+\tilde x\leq 2 \ell_s$ and $\ell_{s-1}-\ell_s\leq \ell_1-x_k.$

Hence, the term $\pi_1$ can be handled in much the same way as in \eqref{pi1}:
$$ \pi_1\geq \frac{d_1(\tilde x)\,d_2(\tilde x)d_k(\tilde x)}{d_1(x_k)\,d_{k-1}(x_k)d_k(x_k)}\geq \frac{ h_{s} h_{s-1}(\ell_1-2\ell_s)}{\ell_{s-1}
(\ell_1-2\ell_{s-1}+2\ell_s)\ell_1}> \frac{ h_{s} h_{s-1}}{\ell_{s-1}\ell_1}\geq  \frac{1}{3} \ell_s= \frac{1}{3} \ell_1^{\alpha^{s-1}}.$$

Likewise, the product $\pi_2$ consists of $2^{s-1}$ terms of the kind $1+\frac{x_k-\tilde x}{x_i-x_k}$, for which $x_k-\tilde x>\ell_1-\ell_{s-1}-\ell_{s}$
and $x_i-x_k\leq 1-x_k\leq 1-\ell_1+\ell_{s-1}$ as $x_i\in I_{2,1}.$ Therefore, the estimate \eqref{pi2} is also true for any uniform distribution.
From this, $ | l_k ({\tilde x})|\geq \frac{1}{3}\ell_1^{\alpha^{s-1}} \left( \frac{1-\ell_{s}}{1-\ell_1+ \ell_{s-1}}\right)^{2^{s-1}},$ which
is as big as we want for large $s$.\\

Now suppose $\alpha > 2.$ Fix $\varepsilon >0.$ We want to get a uniform upper bound for $\lambda_{2^s}(x).$ Fix $x$ in $K^{\alpha}$ and, as in \eqref{x},
the chain $ x\in I_{j,s} \subset  I_{j_1,s-1} \subset \cdots \subset I_{j_s,0}$
with the corresponding adjacent intervals $J_m.$ Since $Z\in \mathcal{U}$, we have $m_{j,s}=1$ and $\nu_m=2^{s-m}$ for $1\leq m \leq s$. We enumerate
$(x_k)_{k=1}^{2^s}$ in order of increasing distance to $x$. Then $x_1\in  I_{j,s}, x_2 \in J_s, \cdots, (x_k)_{k=2^{s-1}}^{2^s}\subset J_1$ and
$$ \lambda_{2^s}(x)=|l_1(x)|+\sum_{m=1}^{s}\sum_{x_k\in J_m}|l_k(x)|$$
with $2^{s-m}$ points $x_k$ in $J_m.$

For the first term we have
$$ |l_1(x)|=\prod \limits_{i=2}^{2^s}\left|\frac{ x - x_i}{x_1 - x_i}\right|=\left|\frac{ x - x_2}{x_1 - x_2}\right| \cdot  \prod \limits_{m=1}^{s-1}\prod \limits_{x_i\in J_m}\left|\frac{ x - x_i}{x_1 - x_i}\right|.$$
Here, $\left|\frac{ x - x_2}{x_1 - x_2}\right|\leq \frac{\ell_{s-1}}{h_{s-1}}\leq 1+2\cdot (3^{\alpha^{s-2}}-2)^{-1},$ by Lemma \ref{llh}.
For $x_i\in J_m$ we have $|x_1-x_i|\geq h_{m-1}$ and  $\left|\frac{ x - x_i}{x_1 - x_i}\right|=\left|1+\frac{ x - x_1}{x_1 - x_i}\right|\leq
1+\frac{\ell_{s}}{h_{m-1}}.$ Hence, $\pi_2:= \prod \limits_{m=1}^{s-1}\prod \limits_{x_i\in J_m}\left|\frac{ x - x_i}{x_1 - x_i}\right|\leq
\prod \limits_{m=1}^{s-1}\left(1+\frac{\ell_{s}}{h_{m-1}}\right)^{2^{s-m}}$ with $\log \pi_2\leq \sum_{m=1}^{s-1}2^{s-m}\frac{\ell_{s}}{h_{m-1}}=
2^{s-1}\,\ell_{s}\,\sum_{m=1}^{s-1}\frac{1}{2^{m-1}\,h_{m-1}}.$ By Lemma \ref{sum}, the sum does not exceed $(2^{s-3}\,h_{s-2})^{-1}$ and
$\log \pi_2\leq 4\cdot \frac{\ell_{s}}{h_{s-2}}< 12\cdot \frac{\ell_{s}}{\ell_{s-2}}=12\cdot \ell_{s-2}^{\alpha^2-1}< 12\cdot 3^{-3\cdot 2^{s-3}}$
as $\alpha > 2.$ Hence,
$$\log  |l_1(x)|\leq 2\cdot (3^{\alpha^{s-2}}-2)^{-1}+12\cdot 3^{-3\cdot 2^{s-3}},$$
which tends to zero as $s\to \infty.$ Given $\varepsilon,$ we choose $s_1$ such that $ |l_1(x)|\leq 1+ \varepsilon$ for $s\geq s_1.$\\

We proceed to estimate $|l_k(x)|$ for $k\geq 2.$ Fix $m\leq s$ and $x_k\in J_m.$ Let $I:=I_{j_{s-m+1},m-1}.$ Then
\begin{equation}\label{lk}
 |l_k(x)|= \prod_{x_i\in I\setminus \{x_k\}} \left|\frac{ x - x_i}{x_k - x_i}\right| \,\,\cdot  \prod \limits_{x_i\notin I}\left|\frac{ x - x_i}{x_k - x_i}\right|.
\end{equation}
The first product has $1+\nu_s+\nu_{s-1}+\cdots+\nu_m-1=2^{s-m+1}-1$ terms, so it is  $\prod_{i=1}^{2^{s-m+1}-1}\frac{d_i(x)}{d_i(x_k)}.$
We see that $d_1(x)\leq \ell_{s}, d_2(x)\leq \ell_{s-1}, d_3(x)$ and $d_4(x)\leq \ell_{s-2}, \cdots$ with $2^{s-m}-1$ points $x_i$ in $J_m\setminus \{x_k\}$
for which $d_i(x)\leq \ell_{m-1}.$
Similarly, $d_1(x_k)\geq h_{s-1}, d_2(x_k), d_3(x_k)\geq h_{s-2}, \cdots.$ Hence,
$$ \prod_{x_i\in I\setminus \{x_k\}} \left|\frac{ x - x_i}{x_k - x_i}\right|= \frac{\ell_{s} \ell_{s-1}\ell_{s-2}^2\cdots\ell_m^{2^{s-m-1}}
\ell_{m-1}^{2^{s-m}-1}}{h_{s-1}h_{s-2}^2\cdots h_{m-1}^{2^{s-m}}}= \frac{\ell_{s}}{\ell_{m-1}}\cdot \pi_1(m)$$
with
$\pi_1(m):=\frac{\ell_{s-1}}{h_{s-1}} \left(\frac{\ell_{s-2}}{h_{s-2}}\right)^2 \cdots \left(\frac{\ell_{m-1}}{h_{m-1}}\right)^{2^{s-m}}\leq
\left(\frac{\ell_{m-1}}{h_{m-1}}\right)^{2^{s-m+1}}$ since the sets  $K^{\alpha}$ is $\ell$-regular.

The second product in \eqref{lk} (denoted $\pi_2(m)$) can be handled in much the same way as  $\pi_2$ above. If $x_i\in J_n$ with $1\leq n\leq m-1$ then
$\left|\frac{ x - x_i}{x_k - x_i}\right|\leq 1+\frac{\ell_{m-1}}{h_{n-1}}$ as $x,x_k\in I$ and $|x_k - x_i|\geq h_{n-1}.$ There are $2^{s-n}$ nodes in
$J_n$, so $\pi_2(m)\leq \prod_{n=1}^{m-1} \left(1+ \frac{\ell_{m-1}}{h_{n-1}}\right)^{2^{s-n}}$. Of course, $\pi_2(1)=1.$

Thus, for  $x_k$ from $J_m$ we have
\begin{equation}\label{lkk}
 |l_k(x)|\leq \frac{\ell_{s}}{\ell_{m-1}}\cdot \pi_1(m) \pi_2(m).
\end{equation}

To estimate $\pi_1(m) \pi_2(m)$ from above, consider various cases of $m$. Let $m\geq 3.$ Then
\begin{equation}\label{pii1}
\log \pi_1(m) \leq 2^{s-m+1}\log(1+\frac{2\ell_m}{h_{m-1}})< 2^{s-m+2}\frac{\ell_m}{h_{m-1}}\leq \frac{2^{s-m+2}}{3^{\alpha^{m-2}}-2},
\end{equation}
by Lemma \ref{llh}. Hence, $\log \pi_1(m) \leq \frac{2^s}{14}$ for such $m, \alpha.$ In its turn,
$\log \pi_2(m) \leq \sum_{n=1}^{m-1} 2^{s-n} \frac{\ell_{m-1}}{h_{n-1}}= 2^{s-1} \ell_{m-1} \sum_{n=1}^{m-1}\frac{1}{2^{n-1}\,h_{n-1}}.$
By Lemma \ref{sum}, $\log \pi_2(m) \leq 14 \cdot 2^{s-m} \frac{\ell_{m-1}}{h_{m-2}}.$ Lemma \ref{llh} now yields
\begin{equation}\label{pii2}
\log \pi_2(m) \leq \frac{14\cdot 2^{s-m}}{ 3^{\alpha^{m-3}}-2}
\end{equation}
 with the maximal value $7\cdot 2^{s-2}.$  Thus, $\log (\pi_1(m)\pi_2(m))\leq 2^{s+1}$ for $m\geq 3.$

For $m=2$ we have $\pi_1(2)\leq \left(\frac{\ell_1}{h_1}\right)^{2^{s-1}}$ with $\log \pi_1(2) \leq 2^{s-1} \log 3$ and
$\pi_2(2)\leq \left(1+\frac{\ell_1}{h_0}\right)^{2^{s-1}}$ with $\log \pi_2(2) \leq 2^{s-1}.$ Here, $\log (\pi_1(2)\pi_2(2))\leq 2^{s+1}$.
The same bound is valid for  $\log \pi_1(1)$ since $\pi_1(1)\leq \left(\frac{\ell_0}{h_0}\right)^{2^s}\leq 3^{2^s}.$ Substituting this into \eqref{lkk}
yields
$$\sum_{m=1}^{s}\sum_{x_k\in J_m}|l_k(x)|\leq \sum_{m=1}^{s}  2^{s-m} \frac{\ell_{s}}{\ell_{m-1}} e^{2^{s+1}}<e^{2^{s+1}}2^{s-1}\ell_{s}\sum_{m=1}^{s} \frac{1}{2^{m-1}\,h_{m-1}}.$$ By Lemma \ref{sum}, $ RHS \leq 6\cdot e^{2^{s+1}}\ell_{s}/\ell_{s-1}\leq 6\cdot e^{2^{s+1}} 3^{-\alpha^{s-2}},$
which is less than $\varepsilon$ for $s\geq s_2$ as
\begin{equation}\label{last}
 2^{s+1}< \alpha^{s-2}\log 3.
\end{equation}
Therefore, $ \lambda_{2^s}(x)<1+2\,\varepsilon$ for $s\geq \max\{s_1,s_2\}$ provided $\alpha>2.$\\

For $\alpha=2$, \eqref{last} is not valid and we need more detailed analysis of $\pi_j(m)$. As above, fix $\varepsilon >0$
and $x \in K^2.$ In the same manner we can see that $|l_1(x)|\leq 1+ \varepsilon$ for $s\geq s_1.$

Let $m_0:=4+\left[\frac{\log s}{\log 2}\right].$ We split the sum $\sum_{m=1}^{s}\sum_{x_k\in J_m}|l_k(x)|$ into several parts.

Suppose $m\geq m_0.$ Then, by \eqref{pi1}, $\log \pi_1(m)<1.$ Similarly,  by \eqref{pi2}, $\log \pi_2(m)<1$ as $\alpha^{m-3}>s.$ Hence, by \eqref{lkk}
and Lemma \ref{sum},
$$\sum_{m=m_0}^{s}\sum_{x_k\in J_m}|l_k(x)|\leq \sum_{m=m_0}^{s} 2^{s-m} e^2 \frac{\ell_s}{\ell_{m-1}}< 6\cdot e^2\,\frac{\ell_s}{\ell_{s-1}}< \varepsilon
\,\,\,\,\,\,\mbox{for}\,\,\, s\geq s_2.$$

If $3\leq m<m_0$ then $\pi_1(m)\leq \left(\frac{\ell_{m-1}}{h_{m-1}}\right)^{2^{s-m+1}}\leq \left(\frac{\ell_2}{h_2}\right)^{2^{s-2}},$ so
$\log \pi_1(m)\leq 2^{s-2} \log\frac{9}{7}.$

If $m\geq 4$ then, by \eqref{pi2}, $\log \pi_2(m)\leq 2^{s-3}.$ Also,  $\log \pi_2(3)\leq \sum_{n=1}^2 2^{s-n} \frac{\ell_2}{h_{n-1}}=\frac{5}{12}\,2^s.$
Therefore, $\log \pi_1(m)\log \pi_2(m)\leq e^{2^{s-1}}$ for such $m$ and similar to above
$$\sum_{m=3}^{m_0-1}\sum_{x_k\in J_m}|l_k(x)|\leq \sum_{m=3}^{m_0-1} 2^{s-m} e^{2^{s-1}}\frac{\ell_s}{\ell_{m-1}}< 12\cdot  2^{s-m_0} e^{2^{s-1}}\frac{\ell_s}{\ell_{m_0-2}}.$$
This does not exceed $\varepsilon$ for $s\geq s_3,$ because $\log\frac{12}{\varepsilon}+s \log 2 + 2^{s-1}<(2^{s-1}-2^{m_0-3})\log 3.$\\

The cases $m=2$ and $m=1$ are more complicated. To simplify the writing, we fix $\ell_1=\frac{1}{3},$ so $h_0=\frac{1}{3}, \ell_2=h_1=\frac{1}{9}, \ell_3=\frac{1}{81}, h_2=\frac{7}{81}.$

For fixed $x_k\in J_2$ we consider the decomposition \eqref{lk} with $I=I_{j_{s-1},1}.$ There is no loss of generality in assuming $I=I_{1,1},$
so $x$ and $x_k$ lie on the left half of $K^2$. As above, we denote the second product of \eqref{lk}, that is
$\prod \limits_{x_i\in I_{2,1}}\left|\frac{ x - x_i}{x_k - x_i}\right|$ by $\pi_2(2).$ Estimating $\pi_2(2)$, we can assume that $x\in I_{1,2}$ and
$x_k\in I_{2,2},$ since otherwise $\pi_2(2)<1.$ The maximum value of $\frac{ x_i - x}{x_i - x_k}$ is reached at the minimum $x$ and maximum $x_k.$
Hence we can take $x=0$ and $x_k=\ell_1.$ The function $\frac{t}{t - \ell_1}$ decreases on $[1-\ell_1, 1].$
We divide $x_i\in I_{2,2}$ into four groups corresponding  subintervals of the third level: $x_i\in I_{5,3}\Rightarrow \frac{ x_i}{x_i - \ell_1}\leq \frac{1-\ell_1}{1- 2\ell_1}=2,$ if $x_i\in I_{6,3}$ then $\frac{ x_i}{x_i - \ell_1}\leq \frac{1-\ell_1+\ell_2-\ell_3}{1-2\ell_1+\ell_2-\ell_3}=\frac{62}{35},$
$x_i\in I_{7,3} \Rightarrow \frac{ x_i}{x_i - \ell_1}\leq \frac{1-\ell_2}{1-\ell_1-\ell_2}=\frac{8}{5}, x_i\in I_{8,3} \Rightarrow \frac{ x_i}{x_i - \ell_1}\leq \frac{1-\ell_3}{1-\ell_1-\ell_3}=\frac{80}{53}.$ Each group contains $2^{s-3}$ nodes. Hence, $\log \pi_2(2)\leq 2^{s-3} \log 8.56.$

The first product in \eqref{lk} can be written as
$\left|\frac{1}{x- x_k}\right|\cdot \frac{\prod_{i=1}^{2^{s-1}} d_i(x) }{\prod_{i=2}^{2^{s-1}} d_i(x)},$ which is smaller than
$\frac{\ell_s}{h_1} \pi_3 \pi_4,$ where, in $\pi_3:=\prod_{i=2}^{2^{s-2}} \frac{d_i(x)}{d_i(x_k)}$, the distances $d_i(y)$ correspond to the nodes in
$I_2\ni y,$ whereas  $d_i(y)$ in $\pi_4:=\prod_{i=2^{s-2}+1}^{2^{s-1}} \frac{d_i(x)}{d_i(x_k)}$ do for the adjacent to $I_2$ interval.
Just like for $\pi_1(m)$ above, we get $\pi_3\leq \frac{\ell_{s-1}}{h_{s-1}} \cdots \left(\frac{\ell_3}{h_3}\right)^{2^{s-4}}  \left(\frac{\ell_2}{h_2}\right)^{2^{s-3}}\leq \left(\frac{\ell_3\,\ell_2}{h_3\,h_2}\right)^{2^{s-3}}$. Hence, $\log \pi_3\leq 2^{s-3}\log 1.32.$

To estimate $\pi_4,$ by symmetry, we assume, as above, that $x=0, x_k\in I_{2,2}.$ Then $d_i(x)\leq d_i(0)=\ell_2+h_1+\delta_i,$ where
$(\delta_i)_{i=1}^{2^{s-2}}$ are distances from points of $Z\cap I_{2,2}$ to the left endpoint of $I_{2,2}$. Likewise,
$d_i(x_k)\geq h_1+\delta'_i,$ where $(\delta'_i)_{i=1}^{2^{s-2}}$ are distances from points of $Z\cap I_{1,2}$ to $\ell_2.$
In both cases we arrange distances increasingly.
Since each interval of the $s-$th level contains exactly one point of $Z$, it follows that $|\delta_i-\delta'_i|\leq \ell_s$ for every $i.$
Hence,
$$\pi_4\leq \prod_{i=1}^{2^{s-2}} \frac{\ell_2+h_1+\delta'_i+\ell_s}{h_1+\delta'_i}\leq \prod_{i=1}^{2^{s-2}} \frac{\ell_2+h_1+\ell_s}{h_1}=
(2+9\ell_s)^{2^{s-2}}<2\cdot 2^{2^{s-2}}.$$
Combining these yields $|l_k(x)|\leq  \frac{\ell_s}{h_1}  \pi_2 \pi_3 \pi_4$ and $\sum_{x_k\in J_2}|l_k(x)|\leq 9\cdot 2^{s-2} \ell_s \pi_2 \pi_3 \pi_4.$
This does not exceed $\varepsilon$ provided $\log\frac{9}{\varepsilon}+(s-2) \log 2 + \log(\pi_2 \pi_3 \pi_4)<2^{s-1}\log 3,$ which is valid for
$s\geq s_4$ since $\log(\pi_2 \pi_3 \pi_4)\leq 2^{s-3}\log 16.$

It remains to consider $m=1$ with $\pi_2(1)=1.$ Here, $|l_k(x)|\leq \left|\frac{\ell_s}{h_0}\right|\cdot \prod_{i=2}^{2^{s}} \frac{d_i(x)}{d_i(x_k)}.$
Now we use the decomposition of
$\prod_{i=2}^{2^{s}}$ into three parts $\prod_{i=2}^{2^{s-2}} \prod_{i=2^{s-2}+1}^{2^{s-1}} \prod_{i=2^{s-1}+1}^{2^{s}},$ where the first product
$\pi_5$ contains distances to the closest nodes (on intervals of the 2-nd level), so it coincides with $\pi_3(2)$; the middle part $\pi_6$
 contains distances to the nodes on adjacent intervals of the 2-nd level, so it equals $\pi_4(2)$, and $\pi_7$ contains $d_i>h_0.$ Hence,
 $\log(\pi_5\,\pi_6)\leq 2^{s-3}\log (4\cdot 1.32) +\log 2.$

The term $\pi_7$ can be handled in much the same way, as $\pi_4$ above. If  $2^{s-1}<i\leq 2^{s-1}+2^{s-2}$ then
 $ \frac{d_i(x)}{d_i(x_k)}\leq  \frac{\ell_1+h_0+\delta_i}{h_0+\delta'_i}\leq  \frac{\ell_1+h_0+\ell_s}{h_0}=2+3\,\ell_s.$
 If  $2^{s-1}+2^{s-2}<i\leq 2^s$ then
 $ \frac{d_i(x)}{d_i(x_k)}\leq  \frac{1-\ell_2+\delta_i}{1-\ell_1-\ell_2+\delta'_i}\leq  \frac{1-\ell_2+\ell_s }{1-\ell_1-\ell_2}=
 \frac{8+9\ell_s}{5}.$ This gives $\pi_7<2\,(16/5)^{2^{s-2}}.$

As above, $\sum_{x_k\in J_1}|l_k(x)|\leq 3\cdot 2^{s-1} \ell_s \pi_5 \pi_6 \pi_7 < \varepsilon$ for $s\geq s_5,$ because
 $\log(\pi_5 \pi_6 \pi_7)\leq 2^{s-3}\log 16.$

Thus, $ \lambda_{2^s}(x)<1+5\,\varepsilon$ for $s\geq \max_{1\leq k\leq 5} s_k,$ which completes the proof.
\end{proof}

\vspace{3 mm}

If we decrease the number of nodes by 1, then the corresponding subsequence of Lebesgue constants is not bounded for each $\alpha$.

\begin{theorem}\label{unif}
Given $\alpha >1$, suppose that for each $s$ the set $Z_s=(x_{k,s})_{k=1}^{2^s-1}$ is uniformly distributed on $K^{\alpha}$.
Then $\Lambda_{2^s-1}(Z_s, K^{\alpha})\to \infty$ as $s\to \infty.$
\end{theorem}
\begin{proof}
Fix $Z_s\in \mathcal{U}$ and drop the subscript $s$ in $x_{k,s}.$ Each interval of the $s-$th level, except one, let $I_{j,s}$, contains one point from $Z_s$.
Fix $\tilde x\in I_{j,s}$ and, as in \eqref{x}, the chain of basic intervals containing this point:
$\tilde x\in I_{j,s}\subset I_{j_1,s-1}\subset\cdots \subset I_{j_s,0}.$ The idea of the proof is to choose a suitable value
$q=q(s)$ with $q(s)\to \infty$ and obtain the estimate $|l_k(\tilde x)|>\varepsilon_0$ uniform in $s$ and $x_k\in I_{j_q,s-q}$.
There are $2^q-1$ interpolating nodes in $I_{j_q,s-q}$. Therefore,
$\Lambda_{2^s-1}(Z_s, K^{\alpha})>\sum_{x_k \in I_{j_q,s-q}}|l_k(\tilde x)|\to \infty$ as $s\to \infty,$ the desired conclusion.

Let $q=[\log s]$. Fix $x_k\in I_{j_q,s-q}.$ We use the decomposition
$|l_k(\tilde x)|=\pi_1 \cdot \pi_2,$ where the product $\pi_1=\prod |\frac{\tilde x - x_i}{x_k - x_i}|$ is taken for
$x_i \in Z \cap I_{j_q,s-q}$ and, correspondingly, $\pi_2$ does for the remaining nodes.

Let us evaluate $\pi_1$. Multiplying both parts of the fraction by $|x_k-\tilde x|$ yields $\pi_1=\frac{|\omega'(\tilde x)|}{|\omega'(x_k)|},$ where
$\omega(x)=\prod_{x_i\in I_{j_q,s-q}}(x-x_i).$ It is easily seen that $|\omega'(x_k)|\leq \ell_{s-1}\,\ell_{s-2}^2\cdots \ell_{s-q}^{2^{q-1}}$
and $|\omega'(\tilde x)|\geq h_{s-1}\, h_{s-2}^2\cdots h_{s-q}^{2^{q-1}}$ with
$\pi_1\geq \frac{h_{s-1}}{\ell_{s-1}} (\frac{h_{s-2}}{\ell_{s-2}})^2\cdots  (\frac{h_{s-q}}{\ell_{s-q}})^{2^{q-1}}$. The value $\frac{h_n}{\ell_n}$
increases, so $\pi_1\geq (\frac{h_{s-q}}{\ell_{s-q}})^{2^q}$ and $\log \pi_1 > 2^q \log(1-2  \ell_{s-q}^{\alpha-1})> -2^{q+2}\ell_{s-q}^{\alpha-1}. $
Here, $2^q\leq s^{\log 2}$ and $\alpha^{s-q-1}\geq \alpha^{s/2}$ for $s\geq 4.$
Therefore,  $\log \pi_1 >-4 s^{\log 2} \ell_{1}^{(\alpha-1)\alpha^{s/2}},$ which tends to zero as $s$ increases. Thus, there exists an
$\varepsilon_1>0,$ depending only on $\alpha$, such that $\pi_1\geq \varepsilon_1.$

To deal with $\pi_2$, we use the argument from Theorem \ref{bdd1}: $\pi_2= \prod \limits_{n=1}^{s-q} \prod_{x_i\in J_n}|\frac{\tilde x-x_i}{x_k - x_i}|,$
where $J_n$ and $I_{j_{s-n},n}$ are adjacent.  The interval $J_n$ contains $2^{s-n}$
interpolating nodes $x_i$. For each of them we have
$|\frac{\tilde x - x_i}{x_k - x_i}|\geq 1- \frac{|\tilde x - x_k|}{|x_k - x_i|}\geq 1-\frac{\ell_{s-q}}{h_{n-1}}$ and
$\pi_2\geq \prod \limits_{n=1}^{s-q} (1- \frac{\ell_{s-q}}{h_{n-1}})^{2^{s-n}}.$ Hence,
$\log\pi_2\geq \sum_{n=0}^{s-q-1} 2^{s-n-1}\log(1- \frac{\ell_{s-q}}{h_n})\geq -2^s\,\ell_{s-q} \sum_{n=0}^{s-q-1}\frac{1}{2^n h_n}.$
By Lemma \ref{sum}, $\log\pi_2\geq -2^{q+3}\,\frac{\ell_{s-q}}{h_{s-q-1}}$ for large enough $s$.
 Thus, $\log \pi_2 >-3\cdot 2^{q+3} \ell_{s-q-1}^{\alpha-1},$ which tends to zero as $s$ increases. We conclude that there exists an
$\varepsilon_2=\varepsilon_2(\alpha)>0$ such that $\pi_2\geq \varepsilon_2,$ which is the desired conclusion.
\end{proof}

\section{\bf Sets $K^{\alpha}$. Arbitrary distribution.}

In this section, we consider a generalization of the previous theorem, at least on small $ K^{\alpha}$, to any set of $2^s-1$ interpolating nodes.

\begin{theorem}\label{arb}
If $\alpha > 2$ then $K^{\alpha} \notin {\mathcal{BLC}}$.
\end{theorem}
\begin{proof}
Suppose, contrary to our claim, that there is $M$ such that for each $N$ one can find a set $Z$ of $N$ points from $K^{\alpha}$ satisfying
$\Lambda_N(Z, K^{\alpha})\leq M.$ We restrict ourselves to the values $N=2^s-1$.
Thus, for each $s$ there is $Z_s=(x_{j,s})_{j=1}^{2^s-1}\subset K^{\alpha}$ such that
\begin{equation}\label{lbdd}
|l_k(x)|\leq M \,\,\,\,\,\mbox {for}\,\,\, 1\leq k \leq 2^s-1,\,\,\,x\in K^{\alpha}.
\end{equation}
For a given $M$, let $s_M$ be such that $\alpha^{s-2}>2^{s+2}$ and $2^s>\log M+1$ for $s\geq s_M.$
In what follows, we consider $s\geq s_M$ and the corresponding set $Z_s.$ As above, we drop the subscript $s$ in $x_{j,s}.$
We can certainly assume that $Z_s \notin {\mathcal{U}},$ since otherwise Theorem \ref{unif} can be applied.
By Lemma \ref{RR}, $R:=R_{1,0}\geq s.$ Fix $I_{j,R}$ with $m_{j,R}=2$ and enumerate the nodes in a such way that
$x_1, x_2 \in I_{j,R}$ and $d_i(x_1)=|x_1-x_i|.$  As in \eqref{x}, let $I_{j,R} \subset I_{j_1,R-1} \subset \cdots \subset I_{j_{R},0}=I_{1,0}$.
We denote $I_{j_{R-n},n}$ briefly by $I_n$.
Then the interval $J_k$ is adjacent to $I_k$  for $1\leq k \leq R$ with $\nu_k:=\#(J_k \cap Z_s).$
Since $J_k$ contains $2^{R-k}$ intervals of the $R-$th level, each of them has at most two points of $Z_s$, we have  $\nu_k\leq 2^{R-k+1}$.

Our goal is to prove step by step that each $J_k$ does not contain empty subintervals of the $R-$th level. If $m_{n,R}=0$ for some
$ I_{n,R} \subset J_k$, then we fix any $\tilde x\in I_{n,R}\cap K^{\alpha}$ and estimate from below
$|l_1(\tilde x)|$  to get a contradiction with \eqref{lbdd}.

Suppose first that $\nu_R=0.$ Then for $\tilde x\in J_R\cap K^{\alpha}$ we have $|l_1(\tilde x)|=\left|\frac{\tilde x - x_2}{x_1 - x_2}\right| \cdot \pi_2,$ where $\pi_2:= \prod_{i=3}^N\left|\frac{\tilde x - x_i}{x_1 - x_i}\right|= \prod \limits_{k=1}^{R-1} \prod_{x_i\in J_k}\left|\frac{\tilde x - x_i}{x_1 - x_i}\right|.$ Analysis similar to that in the proof of Theorem \ref{unif} shows that
 $\pi_2\geq \prod \limits_{k=1}^{R-1} (1- \frac{\ell_{R-1}}{h_{k-1}})^{2^{R-k+1}}.$ As above, we have
$\log\pi_2\geq -\sum_{k=1}^{R-1} 2^{R-k+2}\frac{\ell_{R-1}}{h_{k-1}}>-1,$ by Lemma \ref{sum}. Hence,
$ |l_1(\tilde x)|> \frac{1}{e}\, \frac{h_{R-1}}{\ell_R}\geq \frac{1}{3e}\, \frac{\ell_{R-1}}{\ell_R}=
\frac{1}{3e}\, \left(\frac{1}{\ell_1}\right)^{\alpha^{R-1}-\alpha^{R-2}}\geq \frac{1}{3e}\cdot 3^{\alpha^{R-2}}$ as $\alpha>2$ and
$\ell_1\leq \frac{1}{3}.$ Thus, $ |l_1(\tilde x)|$ exceeds $M$ for given $s$, contrary to \eqref{lbdd}. Therefore, $\nu_R\geq 1.$\\

By decreasing induction on $k$, we assume that $J_k$ does not contain empty subintervals of the $R-$th level for $k=R, \cdots, n+1$ with
$3\leq n \leq R-1.$ Hence,
\begin{equation}\label{nu}
 2^{R-k} \leq \nu_k\leq 2^{R-k+1} \,\,\,\,\,\mbox{for}\,\,\, n+1\leq k\leq R.
\end{equation}

Thus, $I_n$ does not contain empty $I_{i,R}$. Let us show that the same is valid for $J_n.$

To obtain a contradiction, suppose that $J_n\supset I'_{j,R}$ with $m'_{j,R}=0.$ As above, take
$I'_R:=I'_{j,R} \subset I'_{R-1} \subset \cdots \subset I'_{n+1}\subset J_n$. Let $J'_k$ be adjacent to $I'_{k}$  for $n+1\leq k \leq R$ with $\nu'_k:=\#(J'_k \cap Z_s).$ Here and below, in order to distinguish numbers of points from $Z_s$ on intervals of the same level in $I_k$ and
$J_k$, we use the notation $m'$ for subintervals of $J_n.$ We can assume, by the induction hypothesis, that $m'_{i,R}\leq 1$ for each $I'_{i,R}\subset J_n.$
From this $\nu'_k\leq 2^{R-k}$ for $n+1\leq k\leq R.$

Fix $\tilde x\in I'_{R}\cap K^{\alpha}.$ Then
\begin{equation}\label{ell1}
 |l_1(\tilde x)|\geq \frac{h_{R-1}^{\nu'_R} h_{R-2}^{\nu'_{R-1}} \cdots h_n^{\nu'_{n+1}} h_{n-1}^{m_n-1} }
{\ell_R \ell_{R-1}^{\nu_R}\cdots \ell_n^{\nu_{n+1}} \ell_{n-1}^{m'_n} } \cdot \pi_2(n)
\end{equation}
with $m_n:=\#(Z_s\cap I_n)=2+ \sum_{k=n+1}^R \nu_k,\,m'_n:=\#(Z_s\cap J_n)=\sum_{k=n+1}^R \nu'_k,
\pi_2(n):= \prod_{x_i\notin I_{n-1}}\left|\frac{\tilde x - x_i}{x_1 - x_i}\right|.$

The fraction in \eqref{ell1} can be written as

$$ \frac{h_{n-1}}{\ell_R}\cdot \left[ \left( \frac{h_{R-1}}{\ell_{n-1}}\right)^{\nu'_R}\cdots \left( \frac{h_n}{\ell_{n-1}}\right)^{\nu'_{n+1}} \right]
\cdot \left[ \left( \frac{h_{n-1}}{\ell_{R-1}}\right)^{\nu_R}\cdots \left( \frac{h_{n-1}}{\ell_{n}}\right)^{\nu_{n+1}} \right].$$
Terms in the first square brackets are smaller than 1, so the minimum of $\left( \frac{h_{k-1}}{\ell_{n-1}}\right)^{\nu'_k}$ is attained for
the maximal degree, so $\nu'_k=2^{R-k}.$ On the other hand, to get minimum of the second brackets, we have take the minimal value of $\nu_k$ since
$h_{n-1}> \ell_k$ for $k\geq n.$ In this case, $m_n=2+m'_n=2^{R-n}+1.$ This choice of degrees will not disturb the estimate of $\pi_2(n).$

Returning to \eqref{ell1}, we get $ |l_1(\tilde x)|\geq \frac{\ell_{n-1}}{\ell_R} \cdot  \pi_1(n) \pi_2(n),$ where

\begin{equation}\label{ppii}
\pi_1(n)\geq \frac{h_{R-1}}{\ell_{R-1}}  \left(\frac{h_{R-2}}{\ell_{R-2}}\right)^2\cdots \left(\frac{h_{n-1}}{\ell_{n-1}}\right)^{2^{R-n}}>
\left(\frac{h_{n-1}}{\ell_{n-1}}\right)^{2^{R-n+1}}.
\end{equation}

By Lemma \ref{llh}, $\log \pi_1(n)\geq 2^{R-n+1} \log(1-\frac{2}{3^{\alpha^{n-2}}} )> 2^{R-n+1} \log\frac{7}{9}$ since $n\geq 3$ and $\alpha>2.$

Therefore, $\log \pi_1(n)\geq-2^R.$

On the other hand, as for $n=R$, we have
$\pi_2(n)= \prod \limits_{k=1}^{n-1} \prod_{x_i\in J_k}\left|\frac{\tilde x - x_i}{x_1 - x_i}\right|.$ If $n\geq 4$ then
$$ \log \pi_2(n)\geq \sum_{k=1}^{n-1} 2^{R-k+1} \log \left(1- \frac{\ell_{n-1}}{h_{k-1}}\right)\geq \sum_{k=1}^{n-1} 2^{R-k+2} \frac{\ell_{n-1}}{h_{k-1}}
\geq-2^{R-n+3}\geq -2^{R-1},$$
because, by Lemmas \ref{sum} and \ref{llh}, $  \sum_{k=1}^{n-1} \frac{1}{2^{k-1}h_{k-1}}\leq \frac{7}{2^{n-2}h_{n-2}}$ and
$\frac{\ell_{n-1}}{h_{n-2}}\leq \frac{1}{7}$ for $n\geq 4.$

In its turn, $\pi_2(3)= \prod_{x_i\in J_1} \cdot \prod_{x_i\in J_2} \left|\frac{\tilde x - x_i}{x_1 - x_i}\right|$ with smaller then $2^R+2^{R-1}$ terms.
For both products $\left|\frac{\tilde x - x_i}{x_1 - x_i}\right|\geq \frac{1}{3},$ as is easy to check. Hence,
$\log \pi_2(3)\geq -3\cdot 2^{R-1}\log 3>-2^{R+1}.$

Thus, $\log \pi_2(n)\geq -2^{R+1}$ for $n\geq 3.$

It is suffices to show that $ \frac{\ell_{n-1}}{\ell_R} \cdot  \pi_1(n) \pi_2(n)>M $ or $\log \frac{\ell_{n-1}}{\ell_R}> \log M+ 2^R + 2^{R+1}$
for $s\geq s_M$.

The left-hand side here is $(\alpha^{R-1}-\alpha^{n-2}) \log \frac{1}{\ell_1},$ which exceed $\alpha^{R-2}$ for $n\leq R-1.$
Of course, the right-hand side is smaller than $2^R(\log M+3).$ Due to the choice of $s_M$, the desired inequality is satisfied as $R\geq s.$\\

It remains to consider the cases $n\leq 2.$  For $n=2$, by \eqref{ppii}, $\pi_1(2)>\left(\frac{h_1}{\ell_1}\right)^{2^{R-1}}> 3^{-2^{R-1}}.$
Here $\pi_2(2)= \prod_{x_i\in J_1}\left|\frac{\tilde x - x_i}{x_1 - x_i}\right|$ contains $R$ terms with
$\left|\frac{\tilde x - x_i}{x_1 - x_i}\right|\geq \frac{h_0}{\ell_0}\geq \frac{1}{3},$ so
$ |l_1(\tilde x)|\geq \frac{\ell_{1}}{\ell_R} \cdot 3^{-2^{R-1}-2^R},$ which exceeds $M$ since 
$ \frac{\ell_{1}}{\ell_R} \cdot 3^{-2^{R-1}-2^R} = {\ell_1}^{(1-\alpha^{R-1})} \cdot 3^{-3\cdot 2^{R-1}} \geq 3^{(\alpha^{R-1}-3\cdot 2^{R-1}-1)}$, and for $\alpha >2$ (recall that $\ R \geq s$), the right hand side is as large as we wish for large enough $s \in \Bbb N$.

Similarly, if $n=1$ then $|l_1(\tilde x)|\geq \frac{\ell_{0}}{\ell_R} \cdot \pi_1(1)$ and $\pi_1(1)>\left(\frac{h_0}{\ell_0}\right)^{2^{R}}> 3^{-2^R}.$ So $|l_1(\tilde x)|$ exceeds $M$ for large enough $s$.\\

Therefore, the whole interval $I_{1,0}$ does not contain empty subintervals of the $R-$th level and $m_{1,0}\geq 2^R$ , which is impossible since
$R\geq s$ and $m_{1,0}=2^s-1.$
\end{proof}

\section{\bf Equilibrium sets $K(\gamma)$}

The third family of Cantor sets consists of quadratic generalized Julia sets introduced in \cite{WECS}. For the convenience of the reader we repeat
the relevant material. Given sequence $\gamma = (\gamma_s)_{s=1}^{\infty}$ with $0<\gamma_s < 1/4,$ let $\delta_s=\gamma_1\gamma_2 \cdots \gamma_s,\, r_0=1$
and $r_s=\gamma_s r_{s-1}^2$ for $s\in \mathbb{N}$.
Define $P_2(x)=x(x-1),\, P_{2^{s+1}}=P_{2^s}(P_{2^s}+r_s)$ and $ E_s=\{x\in {\mathbb{R}}:\,P_{2^{s+1}} (x)\leq 0\}$ for $s\in \mathbb{N}.$
Then $E_s=\cup_{j=1}^{2^s}I_{j,s}$ and $K(\gamma) := \cap_{s=0}^{\infty} E_s.$ By the construction,
\begin{equation}\label{pr}
|P_{2^s}(x)|\leq r_s \,\,\,\,\,\mbox {for}\,\,\, x\in E_s.
\end{equation}

Here the lengths $\ell_{j,s}$ of the intervals $I_{j,s}$ of the $s-$th level are not the same, but, provided the condition
\begin{equation}\label{gamma}
\gamma_k \leq 1/32\,\,\,\,\,\,\mbox{for}\,\,\,\, \,\,\,k\in {\Bbb N} \,\,\,\,\,\,
\mbox{and} \,\,\,\,\sum_{k=1}^{\infty} \gamma_k <\infty,
\end{equation}
by Lemma 6 in \cite{WECS}, we have
\begin{equation}\label{delta}
 \delta_s < \ell_{j,s} < C_0\,\delta_s \,\,\,\,\,\mbox {for}
\,\,\,\,\,\,\, 1\leq j\leq 2^s,
\end{equation}
where $C_0= \exp( 16\,\sum_{k=1}^{\infty} \gamma_k).$

In addition,  by Lemma 4 in \cite{WECS},
\begin{equation}\label{lll}
\ell_{i,s+1}<4 \gamma_{s+1} \ell_{j,s},
\end{equation}
if $I_{i,s+1}$ is a subinterval of $I_{j,s}$ and for $h_{j,s}:=\ell_{j,s}-\ell_{2j-1,s+1}- \ell_{2j,s+1}$  we have
\begin{equation}\label{hhh}
h_{j,s}> (1-4 \gamma_{s+1})\ell_{j,s}\geq 7/8 \cdot \ell_{j,s}\,\,\mbox{ for all}\,\, j\leq 2^s.
\end{equation}

Remarkable property of a non-polar set $K(\gamma)$ provided \eqref{gamma} is that its equilibrium measure coincides with the Cantor-Hausdorff measures
corresponding to this set, see T.8.3 in \cite{GU}.

Let us fix any basic interval $I_{j,s}$ for $s\geq 1.$ Then one its endpoint $y$ belongs to $Y_{s-1},$ whereas another is from $X_s.$
Since, by Lemma 4 (\cite{WECS}), the polynomial $P_{2^s}$ is convex on $I_{j,s}$, we can slightly modify Lemma 5 (\cite{WECS}):
\begin{lemma}\label{p}
Let $\gamma $ satisfy \eqref{gamma}, $t\in I_{j,s}$ for $1\leq s, 1\leq j \leq 2^s.$ Then
$$C_0^{-1} \, r_s/\delta_s < |P'_{2^s}(t)| \leq r_s/\delta_s.$$
\end{lemma}

Our next goal is to evaluate the Lebesgue constants $\Lambda_N(X, K(\gamma))$ for two cases.\\

\begin{lemma}\label{mon}
Let $\gamma $ satisfy \eqref{gamma} and $Y_{s-1}$ be the interpolating set. Then , for large enough $s$, the polynomial $l_{k, 2^s}$ is  monotone on $I_{j,s}$
for each $j,k \in \{1, \ldots, 2^s\}.$
\end{lemma}
\begin{proof}
The set $Y_{s-1}$  consists of $2^s$ endpoints of basic intervals of levels $\leq s-1.$ Let us fix
$x_k\in I_{k,s}\subset I_{m,s-1}\subset I_{q,s-2}.$ Thus, $l_{k, 2^s}(x)=\frac{P_{2^s}(x)}{(x-x_k) P'_{2^s}(x_k)}.$
Let us assume that $k=2m,$ so $x_k$ is the right endpoint of $I_{m,s-1}$ and $P'_{2^s}(x_k)>0.$  Fix $I_{j,s}$.

First we suppose that $I_{j,s}$ is not a subinterval of $I_{m,s-1}$. Of course, $l'_{k, 2^s}(x)=\frac{P'_{2^s}(x)(x-x_k)-P_{2^s}(x)}{(x-x_k)^2 P'_{2^s}(x_k)}.$
We show that the numerator of this fraction does not change its sign on $I_{j,s}$. Let $x\in I_{j,s}$. By Lemma \eqref{p},
$|P'_{2^s}(x)| >C_0^{-1} \, r_s/\delta_s.$ By \eqref{hhh}, \eqref{lll}, and \eqref{delta},
$|x-x_k|\geq h_{q,s-2}\geq \frac{7}{8} \cdot \ell_{q,s-2}\geq  \frac{7}{8}  \cdot\frac{1}{16}  \frac{\delta_s}{\gamma_{s-1} \gamma_s},$ which exceeds
$C_0 \delta_s,$ as $\gamma_s \to 0.$ Hence, $|P'_{2^s}(x)|\cdot |x-x_k|\geq r_s$ for large enough $s$. On the other hand, by \eqref{pr}, $|P_{2^s}(x)|\leq r_s,$
as desired.

For $x\in I_{k,s}$ we use Taylor's expansion: $P_{2^s}(x)= P'_{2^s}(\xi)(x-x_k)$ with $x<\xi <x_k$. Here,
$|l_{k, 2^s}(x)|= \left|\frac{P'_{2^s}(\xi)}{P'_{2^s}(x_k)}\right|$ increases to 1 by convexity of $P_{2^s}$.

Let $j=k-1.$  Here all three values $P'_{2^s}(x), x-x_k, P_{2^s}(x)$ are negative, so  $l'_{k, 2^s}(x)>0.$

The same proof remains valid for $k=2m-1,$ except the case of the adjacent interval with $j=k+1.$ Now,  $P'_{2^s}(x)$ and $x-x_k$ are positive.
\end{proof}

\begin{theorem}\label{bdd2}
Let $\gamma $ satisfy \eqref{gamma}, then the sequence $(\Lambda_{2^s}(Y_{s-1}, K(\gamma)))_{s=1}^{\infty}$ is bounded.
\end{theorem}
\begin{proof}
Fix $x\in K(\gamma)$ and, as in \eqref{x}, the intervals $x\in  I_{k_0,s}\subset I_{k_1,s-1}\subset \cdots  \subset I_{k_s,0}=[0,1].$
We proceed to estimate $\lambda_{2^s}(x)$ from above.

Let us enumerate $(x_k)_{k=1}^{2^s}$ in order of increasing distance to $x$. As above, $J_{k_i,s-i}$ and $I_{k_i,s-i}$ are adjacent for $0\leq i \leq s-1.$
We notice that $J_{k_i,s-i}$ contains $2^i$  points from $Y_{s-1}$. By Lemma \eqref{mon}, $|l_1(x)|\leq 1.$ Here, $x_2\in  J_{k_0,s}.$ Then $|x-x_2|\geq h_{k_1,s-1}.$ Arguing as in Lemma \eqref{mon},  we have
 $|l_2(x)|=\left|\frac{P_{2^s}(x)}{(x-x_2) P'_{2^s}(x_2)}\right|\leq \frac{C_0\,\delta_s}{h_{k_1,s-1}}.$ Similarly, for each from $2^i$ interpolating points from
$J_{k_i,s-i}$, let $x_m,$ we have $|l_{m}(x)|\leq \frac{C_0\,\delta_s}{h_{k_{i+1},s-i-1}}.$

Combining these inequalities gives $\lambda_{2^s}(x)\leq 1+ C_0\,\delta_s \sum_{i=1}^s \frac{2^{i-1}}{h_{k_i,s-i}}.$ By \eqref{hh} and \eqref{delta},
$h_{k_i,s-i}\geq  \frac{7}{8} \delta_{s-i}.$ Hence, $\lambda_{2^s}(x)\leq 1+ \frac{8}{7}\, C_0\,\sum_{i=1}^s 2^{i-1} \gamma_s \cdots \gamma_{s-i+1}< 1+ \frac{4 C_0}{105},$
by \eqref{gamma}.
\end{proof}

As above, the behavior of the Lebesgue constants changes drastically when we reduce the number of nodes by one. Then the corresponding sequence has at least linear growth.

\begin{theorem}\label{notbdd}
Let $N=2^s-1$ and  $X$ consist of all points $Y_{s-1}$ except one. Then $\Lambda_{N}(X, K(\gamma))> C_0^{-1}\,(N-2).$
\end{theorem}
\begin{proof}
Let $x_k$ be the eliminated node, so $X=Y_{s-1}\setminus \{x_k\}.$ Then  $\lambda_N(x_k)=\sum_{j=1,j\ne k}^N |l_j(x_k)|$ with
$|l_j(x_k)|=\prod \left|\frac{x_k-x_i}{x_j-x_i}\right|,$ where the product is taken for $i\in\{1, \ldots,N\}\setminus \{j,k\}.$
Multiplying both parts of the product by $|x_k-x_j|$ gives the representation  $|l_j(x_k)|=\left|\frac{P'_{2^s}(x_k)}{P'_{2^s}(x_j)}\right|.$
By Lemma \eqref{p}, $|l_j(x_k)|>C_0^{-1},$ which proves the theorem.
\end{proof}

\bibliographystyle{spmpsci}

\end{document}